\documentclass[12pt]{article} 
\usepackage[utf8]{inputenc}
\usepackage{amsmath,amssymb,amsthm,url,hyperref}

\usepackage[linesnumbered,lined,boxed,commentsnumbered]{algorithm2e}

\theoremstyle{plain}
\newtheorem{theorem}{Theorem}

\newtheorem{lemma}[theorem]{Lemma}

\theoremstyle{definition}

\newtheorem{conjecture}[theorem]{Conjecture}
\theoremstyle{remark}
\newtheorem{remark}[theorem]{Remark}

\newcommand{\Oh}[1]{\mathcal{O}\left(#1\right)}
\newcommand\numberthis{\addtocounter{equation}{1}\tag{\theequation}}
\allowdisplaybreaks
\def\modd#1 #2{#1\ \mbox{\rm (mod}\ #2\mbox{\rm )}}

\begin{document}
\begin{center}
\vskip 1cm{\Large\bf 
On the Sum of Squarefree Integers and a Power of Two
}
\vskip 1cm
\large
Christian Hercher\\
Institut f\"{u}r Mathematik\\
Europa-Universit\"{a}t Flensburg\\
Auf dem Campus 1c\\
24943 Flensburg\\
Germany \\
\href{mailto:christian.hercher@uni-flensburg.de}{\tt christian.hercher@uni-flensburg.de} \\
\end{center}

\begin{abstract}
Erd\H{o}s conjectured that every odd number greater than one can be expressed as the sum of a squarefree number and a power of two. Subsequently, Odlyzko and McCranie provided numerical verification of this conjecture up to $10^7$ and $1.4\cdot 10^9$. In this paper, we extend the verification to all odd integers up to $2^{50}$, thereby improving the previous bound by a factor of more than $8\cdot 10^5$. Our approach employs a highly parallelized algorithm implemented on a GPU, which significantly accelerates the process. We provide details of the algorithm and present novel heuristic computations and numerical findings, including the smallest odd numbers $<2^{50}$ that require a higher power of two as all smaller ones in their representation.
\end{abstract}

\section{Introduction}
In 1950,  Erd\H{o}s \cite{Erdos50} proved the existence of an infinite number of positive integers that cannot be expressed as the sum of a prime number and a power of two. Subsequently, on various occasions \cite{Erdos80, Erdos97},  he relaxed the condition on the first summand and proposed the conjecture that every odd number greater than one can be expressed as the sum of a squarefree number and a power of two. Formally:

\begin{conjecture}[Erd\H{o}s]
Let $n>1$ be an odd integer. Then there exist a squarefree integer $m$ and an integer $k$ such that: \[n=m+2^k.\]
\end{conjecture}
This conjecture is listed as Problem~\#11 on Blooms list~\cite{Bloom}.

\subsection*{Computational verification and contributions from Odlyzko and McCranie}

The initial significant computational inquiry into Erd\H{o}s' Conjecture was conducted by Odlyzko, who validated it for all odd integers up to $10^7$, as noted by Granville and  Soundararajan \cite{Granville}. Guy \cite[sec.\ A19]{Guy} stated a further computational validation run by McCranie up to $1.4\cdot 10^9$. These verifications furnished empirical corroboration for the conjecture, albeit with limitations due to the available computing resources at the time. Odlyzko's and McCranie's findings prompted subsequent endeavors to examine the conjecture for larger numbers, yet few such studies are documented.

\subsection*{Connection to Wieferich primes}

In addition to computational efforts, the more theoretical work by Granville and Soundararajan \cite{Granville}, yielded a significant implication of Erd\H{o}s' Conjecture. They demonstrated that if the conjecture holds for all odd integers, then there are infinitely many primes which are not Wieferich primes---those primes $p$ for which $2^p - 1\equiv 1 \pmod{p^2}$. This connection between Erd\H{o}s' Conjecture and the modular properties of primes underscores the relevance of the conjecture to the distribution of primes and to modular arithmetic.

\subsection*{Contribution of this work}
In this paper, we present an extended numerical verification of Erd\H{o}s' Conjecture for all odd integers up to $2^{50}>10^{15}$,which exceeds the bounds set by Odlyzko and McCranie by a factor of over $10^8$ and $8\cdot10^5$, respectively. A highly parallelized GPU-based algorithm was employed to efficiently sieve squarefree numbers and test representations for each odd $n$ up to the stated limit. This approach paves the way for further computational studies and provides a more robust empirical foundation for Erd\H{o}s' Conjecture.

\subsection*{Outline}
In Section~\ref{sec:Heuristics}, we present some heuristics to Erd\H{o}s' Conjecture. In particular, we determine the a~priori probability that for a randomly selected odd integer $n>1$, at least one of the integers $n-2^1$, $n-2^2$, $n-2^3$, and so on up to $n-2^{\ell}$ is squarefree. In Section~\ref{sec:Algos}, we discuss the algorithms used to compute the results presented in Section~\ref{sec:Results}.

\section{Heuristic arguments}\label{sec:Heuristics}
Let $n>1$ be a random odd integer, $2=p_1<p_2<\dots$ the primes, and $A_{k,i}$ be the event $n\equiv 2^k \pmod{p_{i+1}^2}$. Furthermore, let $A_k=\bigcap_{i=1}^{\infty} A_{k,i}^C$ where $A^C$ denotes the complement of the event $A$. Then $A_k$ describes the event of $n-2^k$ being squarefree. The objective is to compute the a~priori probabilities $c_{\ell}:=P\left(\bigcup_{k=1}^{\ell} A_k\right)$ of the events that at least one of the integers $n-2^1, \dots, n-2^{\ell}$ is squarefree.

\subsection*{Computing $c_{\ell}$ for $\ell\leq 6$}
From the principle of inclusion and exclusion, we can derive the following: 
\begin{align*}
c_{\ell}&=P\left(\bigcup_{k=1}^{\ell} A_k\right)\\
&=\sum_{\emptyset \neq I \subseteq \{1,\dots,\ell\}} (-1)^{|I|-1} P\left(\bigcap_{k\in I} A_k\right)\\
&=\sum_{\emptyset \neq I \subseteq \{1,\dots,\ell\}} (-1)^{|I|-1} P\left(\bigcap_{k\in I} \bigcap_{i=1}^{\infty} A_{k,i}^C\right)\\
&=\sum_{\emptyset \neq I \subseteq \{1,\dots,\ell\}} (-1)^{|I|-1} P\left(\bigcap_{i=1}^{\infty} \left(\bigcap_{k\in I}  A_{k,i}^C\right)\right)\\
&=\sum_{\emptyset \neq I \subseteq \{1,\dots,\ell\}} (-1)^{|I|-1} P\left(\bigcap_{i=1}^{\infty} \left(\bigcup_{k\in I}  A_{k,i}\right)^C\right).\\
\intertext{If we assume, for the purposes of our heuristic, that congruences modulo different primes (and modulo powers of different primes) are stochastically independent, we further obtain the following:}
&=\sum_{\emptyset \neq I \subseteq \{1,\dots,\ell\}} (-1)^{|I|-1} \prod_{i=1}^{\infty} P\left(\left(\bigcup_{k\in I}  A_{k,i}\right)^C\right)\\
&=\sum_{\emptyset \neq I \subseteq \{1,\dots,\ell\}} (-1)^{|I|-1} \prod_{i=1}^{\infty} \left(1-P\left(\bigcup_{k\in I}  A_{k,i}\right)\right). \numberthis \label{eqn:cellPIE}
\intertext{Clearly, we have $P(A_{k,i})=\frac{1}{p_{i+1}^2}$. Furthermore, if no two indices, $k_1\neq k_2$, in $I$ satisfy $2^{k_1} \equiv 2^{k_2} \pmod{p_{i+1}^2}$ then $P\left(\bigcup_{k\in I}  A_{k,i}\right)=\frac{|I|}{p_{i+1}^2}$. This is true for all values of $i$, provided that $I\subseteq \{1,\dots,6\}$. Therefor, for $\ell\leq 6$, we can further simplify to}
c_{\ell}&=\sum_{\emptyset \neq I \subseteq \{1,\dots,\ell\}} (-1)^{|I|-1} \prod_{i=1}^{\infty} \left(1-\frac{|I|}{p_{i+1}^2}\right)\\
&=\sum_{m=1}^{\ell} \binom{\ell}{m} \cdot (-1)^{m-1} \cdot \prod_{i=1}^{\infty} \left(1-\frac{m}{p_{i+1}^2}\right).
\end{align*}

Once the products $d_m:=\prod_{i=1}^{\infty} \left(1-\frac{m}{p_{i+1}^2}\right)$ have been computed, numerical expressions for $c_1$ to $c_6$ are obtained, as presented in Table~\ref{Table_c_ell_ell<=6}. These values are in agreement with those presented by Granville and Soundararajan \cite{Granville}.

\begin{table}[h]
\begin{center}
\begin{tabular}{r|l|l|l}
$\ell$ & $d_{\ell}$ & $c_{\ell}$ & $1-c_{\ell}$\\
\hline 
 1 & $0.810569\dots$ & $0.810569\dots$ & $1.894\ldots \cdot 10^{-1}$\\
 2 & $0.645268\dots$ & $0.975870\dots$ & $2.412\ldots \cdot 10^{-2}$\\
 3 & $0.501948\dots$ & $0.997851\dots$ & $2.148\ldots \cdot 10^{-3}$\\
 4 & $0.378599\dots$ & $0.999860\dots$ & $1.390\ldots \cdot 10^{-4}$\\
 5 & $0.273345\dots$ & $0.999993\dots$ & $6.852\ldots \cdot 10^{-6}$\\
 6 & $0.184435\dots$ & $0.999999\dots$ & $2.669\ldots \cdot 10^{-7}$\\
\end{tabular}
 \caption{Table of heuristically derived values for a~priori probabilities of at least one of the numbers $n-2^1, \dots, n-2^{\ell}$ being squarefree}
\label{Table_c_ell_ell<=6}
\end{center}
\end{table}

\subsection*{Computing $c_{\ell}$ for $7\leq \ell \leq 20$}
Nevertheless, $\ell\geq 6$, this methodology is no longer viable. Given that $2^6-1=3^2\cdot 7$ we have $n\equiv 2^{k+6} \pmod{3^2} \iff n\equiv 2^k \pmod{3^2}$, hence $A_{k+6,1}=A_{k,1}$. With $2^{20}-1=3\cdot 5^2 \cdot 11 \cdot 31 \cdot 41$ being the next such case, we can conclude that for $7\leq \ell \leq 20$ for all $i\geq2$ we can conclude $P\left(\bigcup_{k\in I}  A_{k,i}\right)=\frac{|I|}{p_{i+1}^2}$, as previously demonstrated. However, for $i=1$, we obtain  $P\left(\bigcup_{k\in I}  A_{k,1}\right)=\frac{|I \text{ mod } 6|}{3^2}$, where \mbox{$I \text{ mod } 6:=\{k \text{ mod } 6 \mid k \in I\}$}. Using the abbreviation $d^{\prime}_m:=\prod_{i=2}^{\infty} \left(1-\frac{m}{p_{i+1}^2}\right)$ and equation~\eqref{eqn:cellPIE} for $\ell \leq 20$ we obtain the following result:

\begin{align*}
c_{\ell}&=\sum_{\emptyset \neq I \subseteq \{1,\dots,\ell\}} (-1)^{|I|-1} \prod_{i=1}^{\infty} \left(1-P\left(\bigcup_{k\in I}  A_{k,i}\right)\right)\\
&=\sum_{\emptyset \neq I \subseteq \{1,\dots,\ell\}} (-1)^{|I|-1} \cdot \left(1-\frac{|I \text{ mod } 6|}{9}\right)\cdot d^{\prime}_{|I|}.\numberthis \label{eqn:cell>6}
\end{align*}

\subsection*{Computing $d^{\prime}_k$}

It is well-known that $\prod_p \left(1-\frac{1}{p^2}\right)=\frac{1}{\zeta(2)}=\frac{6}{\pi^2}$. Therefore, we can conclude that $d_1=\frac{8}{\pi^2}$ and $d^{\prime}_1=\frac{9}{\pi^2}$.In the case of $k\geq 2$, it is necessary to compute the product $d^{\prime}_k$ numerically.  In order to ensure good convergence, the following lemma should be observed.

\begin{lemma}\label{Lem:logTaylor}
Let $0<x\leq \frac{1}{9}$ be a real number and $2\leq k \leq 20$ be an integer with $kx \leq \frac{4}{5}$. Then
\[\log(1-kx) = k\cdot \log(1-x) + \binom{k}{2} \log(1-x^2) + e(x) \text{ with } |e(x)|<18x^3.\]
Here and in the following $\log$ denotes the natural logarithm.
\end{lemma}

\begin{proof}
From the Taylor expansion of $\log(1-x)$ we get the existence of $0<\xi<x$ with 
\begin{align*}
\log(1-x)&=-x-\frac{1}{2}x^2-\frac{1}{3(1-\xi^3)}x^3.\\
\intertext{Thus, }
-x-\frac{1}{2}x^2&>\log(1-x)>-x-\frac{1}{2}x^2-\frac{243}{728}x^3.
\intertext{In the same way, we get the existence of another $0<\xi<x$ with }
\log(1-x^2)&= -x^2-\frac{1}{2(1-\xi^2)}x^2.
\intertext{Hence,}
-x^2&>\log(1-x^2)>-x^2-\frac{81}{160}x^4\\
&\geq -x^2-\frac{9}{160}x^3.
\intertext{A third inequality of the same type is derived with}
-kx-\frac{k^2}{2}x^2 &> \log(1-kx)>-kx-\frac{k^2}{2}x^2-\frac{125}{183}x^3\\
&>-kx-\frac{k^2}{2}x^2-x^3.
\intertext{By combining these inequalities, we obtain the following result:}
k\cdot \log(1-x) + \binom{k}{2} \log(1-x^2) -x^3&< -kx-\frac{k^2}{2}x^2 -x^3 < \log(1-kx) 
\intertext{and}
k\cdot \log(1-x) + \binom{k}{2} \log(1-x^2) &> -kx-\frac{k^2}{2}x^2-\left(\frac{243}{728} \cdot k+\frac{9}{160} \cdot \binom{k}{2}\right)x^3 \\
&> \log(1-kx)-\left(\frac{243}{728} \cdot 20+\frac{9 \cdot 190}{160}\right)x^3 \\
&>  \log(1-kx)-18x^3.
\end{align*}

The combination of all the inequalities yields the desired result.
\end{proof}

\begin{lemma}\label{Lem:dkprime}
Let $2\leq k\leq 20$ and $m>5$ be integer. Moreover, for $a\in\{2,4\}$ let $P_{a,m}$ denote the product $P_{a,m}:=\prod_{5\leq p\leq m} \left(1-\frac{1}{p^a}\right)^{-1}$, where the product runs through all primes $5\leq p\leq m$. Then
\[d^{\prime}_k= \left(\frac{9}{\pi^2}\right)^k \cdot \left(\frac{486}{5 \cdot \pi^4} \right)^{\binom{k}{2}} \cdot P_{2,m}^k \cdot P_{4,m}^{\binom{k}{2}} \cdot \prod_{5\leq p \leq m} \left(1-\frac{k}{p^2}\right) \cdot f, \text{ where } |1-f|< 4m^{-5}.\]
\end{lemma}

\begin{proof}
As previously indicated, it is $\prod_p \left(1-\frac{1}{p^2}\right)=\frac{1}{\zeta(2)}=\frac{6}{\pi^2}$, and therefore, $\prod_{p\geq 5} \left(1-\frac{1}{p^2}\right)=\frac{9}{\pi^2}$.Similarly, we have $\prod_p \left(1-\frac{1}{p^4}\right)=\frac{1}{\zeta(4)}=\frac{90}{\pi^4}$ and $\prod_{p\geq 5} \left(1-\frac{1}{p^4}\right)=\frac{90}{\pi^4} \cdot \frac{16}{15} \cdot \frac{81}{80}=\frac{486}{5\cdot \pi^4}$. Using these identities and $d^{\prime}_k=\prod_{p\geq 5} \left(1-\frac{k}{p^2}\right)$ the equation to show can be equivalently stated as follows:
\begin{align*}
\prod_{p> m} \left(1-\frac{k}{p^2}\right) &= \prod_{p> m} \left(1-\frac{1}{p^2}\right)^k \cdot \prod_{p> m} \left(1-\frac{1}{p^4}\right)^{\binom{k}{2}} \cdot f
\intertext{or}
\sum_{p>m} \log \left(1-\frac{k}{p^2}\right) &= \sum_{p> m} \log \left(   k\cdot \left(1-\frac{1}{p^2}\right) + \binom{k}{2} \cdot \log \left(1-\frac{1}{p^4}\right) \right) +  \log(f).
\end{align*}
Given that $p>m>5$, it follows that $0<\frac{1}{p^2}\leq \frac{1}{25}$ and $k \cdot \frac{1}{p^2} \leq \frac{20}{25}=\frac{4}{5}$. Thus, for all $p>m$ from Lemma~\ref{Lem:logTaylor} we obtain 
\[\log\left(1-\frac{k}{p^2}\right) = k\cdot \log\left(1-\frac{1}{p^2}\right) + \binom{k}{2} \log\left(1-\frac{1}{p^4}\right) + e_p \text{ with } |e_p|<18 \cdot p^{-6}.\]
The sum of these equations for all $p>m$ yields
\begin{align*}
\log(f) &= \sum_{p>m} e_p, 
\intertext{hence}
|\log(f)| &\leq \sum_{p>m} |e_p| < 18 \cdot \sum_{p>m} p^{-6} < \int_{x=m}^{\infty} x^{-6} \text{d}x = \frac{18}{5} m^{-5} < 4 m^{-5}. 
\end{align*}
Consequently, on the one hand, this results in the inequality $f>\exp\left(-4m^{-5}\right)>1-4m^{-5}$. On the other hand, we observe that for the function $h(x):=\exp(x)$ it is $h(0)=1$ and $h^{\prime}(x)<\frac{10}{9}$ for all $0<x<0.1$, thus for these $x$ it is $h(x)<h(0)+\frac{10}{9} \cdot x$.  Since $m>5$ it is $4m^{-5}<m^{-4}<\frac{4}{3125}<0.1$. Therefore, $f<\exp\left(\frac{18}{5}m^{-5}\right)<1+\frac{10}{9} \cdot \frac{18}{5}m^{-5}=1+4m^{-5}$.
\end{proof}

In order to achieve sufficient precision in the computation of these tasks, a Sagemath worksheet was utilized, and the formulas presented in equation~\eqref{eqn:cell>6} and Lemma~\ref{Lem:dkprime} were implemented. All numerical evaluations were conducted with 40 significant digits and an input value of  $m=10^7$. The results are presented in Table~\ref{Table_c_ell_ell<=20}. Let $c_0:=0$. Then, for $\ell\geq 1$, the difference $c_{\ell}-c_{\ell-1}$ provides the a~priori probability that $n-2^k$ is squarefree, while for all $1\leq i < k$ the integers $n-2^i$ are not. These values are also presented in Table~\ref{Table_c_ell_ell<=20}.

 \begin{table}[h]
\begin{center}
\begin{tabular}{r|r|r}
$\ell$ & $1-c_{\ell}$ & $c_{\ell}-c_{\ell-1}$\\
\hline
1 & $1.8943 \ldots \cdot 10^{\phantom{0}-1}$ & $8.1056\ldots \cdot 10^{\phantom{0}-1}$ \\
2 & $2.4129 \ldots \cdot 10^{\phantom{0}-2}$ & $1.6530\ldots \cdot 10^{\phantom{0}-1}$ \\
3 & $2.1482 \ldots \cdot 10^{\phantom{0}-3}$ & $2.1980\ldots \cdot 10^{\phantom{0}-2}$ \\
4 & $1.3902 \ldots \cdot 10^{\phantom{0}-4}$ & $2.0092\ldots \cdot 10^{\phantom{0}-3}$ \\
5 & $6.8527 \ldots \cdot 10^{\phantom{0}-6}$ & $1.3216\ldots \cdot 10^{\phantom{0}-4}$ \\
6 & $2.6694 \ldots \cdot 10^{\phantom{0}-7}$ & $6.5857\ldots \cdot 10^{\phantom{0}-6}$ \\
7 & $4.9973 \ldots \cdot 10^{\phantom{0}-8}$ & $2.1696\ldots \cdot 10^{\phantom{0}-7}$ \\
8 & $2.5127 \ldots \cdot 10^{\phantom{0}-9}$ & $4.7460\ldots \cdot 10^{\phantom{0}-8}$ \\
9 & $8.5032 \ldots \cdot 10^{-11}$ & $2.4277\ldots \cdot 10^{\phantom{0}-9}$ \\
10 & $2.2313 \ldots \cdot 10^{-12}$ & $8.2801\ldots \cdot 10^{-11}$ \\
11 & $4.8076 \ldots \cdot 10^{-14}$ & $2.1832\ldots \cdot 10^{-12}$ \\
12 & $8.7820 \ldots \cdot 10^{-16}$ & $4.7197\ldots \cdot 10^{-14}$ \\
13 & $1.5662 \ldots \cdot 10^{-16}$ & $7.2190\ldots \cdot 10^{-16}$ \\
14 & $4.0728 \ldots \cdot 10^{-18}$ & $1.5222\ldots \cdot 10^{-16}$ \\
15 & $7.3484 \ldots \cdot 10^{-20}$ & $3.9994\ldots \cdot 10^{-18}$ \\
16 & $1.0729 \ldots \cdot 10^{-21}$ & $7.2411\ldots \cdot 10^{-20}$ \\
17 & $1.3401 \ldots \cdot 10^{-23}$ & $1.0595\ldots \cdot 10^{-21}$ \\
18 & $1.4726 \ldots \cdot 10^{-25}$ & $1.3254\ldots \cdot 10^{-23}$ \\
19 & $2.5579 \ldots \cdot 10^{-26}$ & $1.2168\ldots \cdot 10^{-25}$ \\
20 & $4.2068 \ldots \cdot 10^{-28}$ & $2.5158\ldots \cdot 10^{-26}$ \\
\end{tabular}
 \caption{Table of heuristically derived values for a~priori probabilities of none of the numbers $n-2^1, \dots, n-2^{\ell}$ being squarefree}
\label{Table_c_ell_ell<=20}
\end{center}
\end{table}

For the sake of simplicity, we may assume that in all cases, $n-2^k$ is squarefree for a $1\leq k\leq21$. (Only in $<5\cdot 10^{-28}$ of the cases an exponent $>20$ is necessary.) This assumption allows us to conclude that the expected smallest needed exponent of $k=1.215854247598\dots$ Therefore, the expected sum of the smallest exponents required for all odd integers $1<n<S$ is given by $E_S:=0.607927123799\ldots \cdot S$. The standard deviation in a single instance is $0.476\ldots$ and for the interval up to $S$, the standard deviation is given by $\sigma_S:=0.336\ldots\cdot \sqrt{S}$.

\begin{remark}
If the congruences $n\equiv 2^k\pmod{p_{k+1}^2}$ for $1\leq k \leq \ell$ are all true for an odd integer $n$, then the smallest exponent~$k$ for which $n-2^k$ is squarefree is at least $\ell+1$. The Chinese Remainder Theorem ensures the existence of such integers~$n$. Consequently, arbitrarily large exponents are required when scanning through increasingly larger areas of the positive integers for~$n$.
\end{remark}

\section{Algorithms}\label{sec:Algos}
In designing and implementing the code, we adhered to two fundamental principles. The first principle is that a given task should not be recalculated multiple times. Secondly, in order to utilize the parallelism that a modern GPU is capable of, it is necessary to avoid diverging branching as much as possible. This is because the parallelism is achieved through a SIMD implementation (Same Instruction Multiple Data). The first approach culminates in the sieving method. The method of sieving involves the identification of squarefree numbers within large intervals of numbers. The second approach then performs the search for representations of numbers $n$ as the sum of a squarefree number and a power of two, in parallel.

\subsection*{The sieve to identify squarefree numbers} 
In the PC setup given to the author, the internal memory of the GPU is used to screen odd positive integers for squarefree numbers in intervals of length $2^{30}\approx 10^9$. The search was restricted to odd positive integers, as for odd integers~$n$, the difference $n-2^k$ is odd, too. (The only exception would be for $k=0$, which was excluded.) The identification of whether an integer in this interval is squarefree or not can be achieved with great efficiency via sieving.

\medskip

\begin{algorithm}[H]
 \SetAlgoLined
  \KwIn{$start$ and $end$ value of interval to screen}
  \KwOut{for every odd $n$ in $[start; end[$ the value $true$/ $false$, whether $n$ is squarefree or not}
  \ForAll{odd integers $n$ in $[start; end[$}{Set $is\_squarefree(n) \leftarrow true$\;}
  \BlankLine
    \ForAll{primes $3\leq p <\sqrt{end}$}
    {
          find smallest odd $m\geq n$ with $m \equiv 0 \pmod{p^2}$\;
          \ForAll{$0\leq i \leq \left\lfloor\frac{end-m}{2 \cdot p^2}\right\rfloor$}
          {
              Set $is\_squarefree(m + 2i\cdot p^2) \leftarrow false$\;
          }
    }
    \Return{$is\_squarefree$}\;
\caption{Identifying all squarefree numbers via sieving}
\label{Algo_Sieving_sqf_numbers}
\end{algorithm}

\medskip

The first and the inner \textbf{forall} loop in the second loop each can be executed fully in parallel since the instructions are independent of one another. Moreover, there are no diverging branches. Accordingly, this task is optimally suited for the SIMD architecture of modern GPUs. If $\ell=end-start$ is defined as the length of the interval, then the first loop has a time complexity of $\Oh{\ell}$ and the second has a time complexity of $\sum_{3\leq p<\sqrt{end}} \left(\Oh{1}+\Oh{\frac{\ell}{p^2}}\right)=\Oh{\pi(\sqrt{end})}+\Oh{\ell}=\Oh{\frac{\sqrt{end}}{\log(end)}+\ell}$, where $p$ always denotes a prime and $\frac{1}{9}\leq \sum_{3\leq p<\sqrt{end}} \frac{1}{p^2} < \sum_{k=1}^{\infty} \frac{1}{k^2}=\zeta(2)=\frac{\pi^2}{6}<\infty$. This approach allows us to ascertain the status of each $n$ in the specified interval, distinguishing between squarefree and non-squarefree elements, in a time complexity of $\Oh{\frac{\sqrt{end}}{\log(end)}+\ell}$.
  
\subsection*{Finding the representation as sum of a squarefree number and a power of two}
Similarly, for every odd $n$ in a given interval the smallest positive integer $k$ such that $n-2^k$ is squarefree can be identified. By restricting $k$ to be positive, it follows that $n-2^k$ is always odd.

\medskip

\begin{algorithm}[H]
  \SetAlgoLined
  \SetKw{KwNot}{not}
  \SetKw{KwAnd}{and}
  \SetKw{KwPrint}{print}
  \KwIn{$start$ and $end$ value of interval to screen; $k_{max}$ as a search limit}
  \KwOut{for every odd $n$ in $[start; end[$ the smallest positive integer~$k$ with $n-2^k$ being squarefree}
  \ForAll{$n$ in $[start; end[$}
  {
      Set $k \leftarrow 1$\;
      \While{$k\leq k_{max}$ \KwAnd\KwNot $is\_squarefree(n-2^k)$}
      {
          Set $k \leftarrow k+1$\;
      }
      \If(no representation found){$k>k_{max}$}
      {
            \KwPrint No representation found for $n=s+2^k$ with squarefree~$s$ and $1\leq k\leq k_{max}$\;
       }
      Set $no\_of\_trials(n) \leftarrow k$\;
  }
  \Return{$no\_of\_trials$}\;
\caption{Finding representation of $n=s+2^k$ with squarefree~$s$ and smallest positive $k$}
\label{Algo_find_representation}
\end{algorithm}

\medskip

This  employs the results from Algorithm~\ref{Algo_Sieving_sqf_numbers}, wherein the status of each odd integer is determined to be  squarefree or not within the same interval utilized in the computation of the current Algorithm~\ref{Algo_find_representation}, as well as the preceding interval. This allows for the search for $k$ up to $k_{\max}=\lfloor\log_2(\ell)\rfloor$. Nevertheless, during our computations, the scenario of all numbers $n-2^k$ with $k\leq k_{\max}$ being non-squarefree did not occur on any occasion. 

From the heuristic perspective, it is reasonable to expect that, on average, the \textbf{while} loop will be executed $\Oh{1}$ times for each value of $n$. This yields an expected time complexity of Algorithm~\ref{Algo_find_representation} of $\Oh{\ell}$. Furthermore, as the computations for different values of $n$ are independent of one another, they can be performed in parallel. The sole aspect of the algorithm that may potentially have diverging branches is the condition of the \textbf{while} loop: If some threads of a thread group have completed the loop because they have found a representation, they can only succeed if the other threads have also concluded there computation. Consequently, the computation within a thread group (typically comprising 32 threads in the CUDA architecture) can only proceed after the \textbf{while} loop has been completed by all of its threads. This results in the serialization of some parts of the computation. Nevertheless, in the given context, this does not significantly impede the computation, as in the majority of cases the loop is only executed a few times.

\subsection*{Collecting the results}
The application of Algorithms~\ref{Algo_Sieving_sqf_numbers} and \ref{Algo_find_representation} has enabled the discovery of a representation for every odd n in the current investigation interval as $n=s+2^k$, where s is a squarefree integer and k is the smallest positive integer. However, these data are physically present in the memory on the GPU and, thus, cannot be used directly in further computations on the CPU. However, it is not necessary to have all of the details; an aggregate is sufficient. Consequently, this can be computed in parallel as well. Our primary interest lies in the maximal exponent~$k$ required for an odd integer~$n$ in the interval, as the overall sum of these exponents $k$ needed for all odd $n$ in the interval collectively.Both the maximum and the sum can be computed efficiently via the concept of parallel reduction, \cite{Harris}. This approach allows us to transfer only single variables, rather than entire arrays of data, thereby avoiding any potential bottlenecks due to data transfer between the GPU and CPU. 

\section{Results}\label{sec:Results}

Utilizing these concepts and algorithms, we developed a C++/CUDA program to identify the smallest positive integer~$k$ such that $n=s+2^k$, where $s$ is a squarefree integer and $n$ is an odd integer with $1<n<2^{50}$. As a consequence of the fact that we found that $k\leq 13$ for each of the relevant $n$, we can conclude that:
\begin{theorem}
Let $1<n<2^{50}$ be an odd integer. Then there exists a squarefree positive integer~$s$ and a positive integer $1\leq k \leq 13$ with \[n=s+2^k.\]
\end{theorem}

Consequently, Erd\H{o}s' Conjecture is now numerical verified to a significant greater extend.

The computations were performed on a desktop PC with an Intel i9 CPU of the 11th generation and a Geforce RTX 3070 GPU from NVIDIA. The entire computation required approximately 120 hours.

During this process, we gathered a wealth of data, which led us to identify the smallest odd integers~$n$ for which $n-2^k$ is not squarefree for all $1\leq k\leq m$ with $m\leq 12$. These values are presented in Table~\ref{Table_smallest_n} as well as in the OEIS sequence A377587.

\begin{table}[h]
\begin{center}
\begin{tabular}{r|r}
$m$ & smallest odd $n$\\
\hline 
 1 &   11\\
 2 &   29\\
 3 & 533\\
 4 & 849\\
 5 & 434977\\
 6 & 10329791\\
 7 & 28819433\\
 8 & 129747557\\
 9 & 6915752957\\
 10 & 2569472629649\\
 11& 23373845739407\\
 12 & 60690478781437
\end{tabular}
 \caption{Table of smallest odd $n$ with $n-2^1,\dots,n-2^m$ all not squarefree}
\label{Table_smallest_n}
\end{center}
\end{table}

The sum of the smallest positive exponents~$k$ required for $n-2^k$ to be squarefree for all odd $1<n<2^{50}$ is
\begin{align*}
k_{\text{sum}}&=684465092067182
\intertext{From Section~\ref{sec:Heuristics} we would expect }
E_{2^{50}}&=0.607927123799\ldots \cdot 2^{50}\\
&\approx 684465092052491.
\intertext{Accordingly, the difference between the observed and from the heuristic anticipated values is merely} 
k_{\text{sum}}-E_{2^{50}}&\approx \phantom{ 6844650920}14690,
\end{align*}
 or about $1.3\cdot 10^{-3}$ standard deviations. Consequently, the heuristic offers an exceptionally precise approximation.

For the interval $1<n<2^{30}$, each $k$ was recorded as the smallest positive exponent with $n-2^k$ being squarefree. The results are accompanied by the expected values of occurrences as calculated by the heuristics from Section~\ref{sec:Heuristics}, the standard deviations of these random experiments, and the difference between the real and expected values in absolute terms and relative to the standard deviations. These are presented in Table~\ref{Table_k_statistics}. In general, it can be observed that for small~$k$ the actual findings on the frequency of occurrence of each value of $k$ are better described by the expected values derived from the heuristic approach than by the results of random experiments with the specified probabilities.

\begin{table}[h]
\begin{center}
\begin{tabular}{r|r|r|r|r|r}
$k$ & \# $n$ with smallest exp.{} $k$ & expected value & $\sigma$& $|\Delta|$ & $\frac{|\Delta|}{\sigma}$\\
\hline 
 1 & 435171212  & 435171170.1 &9079.3& 41.8 & 0.0046\\
 2 &  88745588  &  88745444.2 &8606.7 & 143.7& 0.0167\\
 3 &  11800589 &  11800957.9 &3397.3 & 368.9 & 0.1085\\
 4 & 1078868 & 1078702.6 &1037.5& 165.3&         0.1593\\
 5 & 71032 & 70958.0 & 266.3& 73.9&                    0.2777\\
 6 & 3473 & 3535.7 &59.4 &  62.7&                         1.0546\\
 7 & 122 & 116.4 &10.7 &  5.5&                               0.5111\\
 8 & 24 & 25.4 &5.0 &  1.4&                                     0.2932\\
 9 & 3 & 1.3& 1.1& 1.7&                                           1.4860\\
 10 & 0 & 0.0 & 0.2& 0.0&                                        0.2108  
\end{tabular}
 \caption{Table of statistic on smallest exponents $k$ with $n-2^k$ being squarefree for all odd $1<n<2^{30}$}
\label{Table_k_statistics}
\end{center}
\end{table}

\bigskip
\hrule
\bigskip

\noindent 
2020 \emph{Mathematics Subject Classification}:~Primary 11A67, 11Y99.

\medskip

\noindent 
\emph{Keywords}:~sum of squarefree numbers and powers of two. 

\bigskip
\hrule

\begin{thebibliography}{WWW}
\bibitem{Bloom} T. Bloom, \url{https://www.erdosproblems.com/go_to/11} (2024).
\bibitem{Erdos50} P. Erd\H{o}s, On integers of the form $2^k + p$ and some related problems, \textit{Summa. Brasil. Math.} \textbf{2} (1950), 113--123.
\bibitem{Erdos80} P. Erd\H{o}s and R. Graham, \textit{Old and New Problems and Results in Combinatorial Number Theory}, Monographies de L'Enseignement Mathematique, 1980.
\bibitem{Erdos97} P. Erd\H{o}s, Problems in number theory, \textit{New Zealand J. Math.} (1997), 155--160.
\bibitem{Granville} A. Granville and K. Soundararajan, A binary additive Problem of Erd\H{o}s and the order of 2 mod $p^2$, \textit{Ramanujan J.} \textbf{2} (1998), 283--298.
\bibitem{Guy} R. K. Guy, \textit{Unsolved Problems in Number Theory}, 3rd edition, Springer, 2004.
\bibitem{Harris} M. Harris, Optimizing parallel reduction in CUDA, \textit{Nvidia developer technology}, \textbf{4} (2007), 70, \url{https://developer.download.nvidia.com/compute/DevZone/C/html/C/src/reduction/doc/reduction.pdf}.
\end{thebibliography}
\end{document}